\documentclass[reqno,a4paper, 11pt]{amsart}

\usepackage[utf8]{inputenc}
\usepackage[T1]{fontenc}

\usepackage{amsthm,amssymb,amsmath, amscd, bm, comment, epsfig, enumitem, verbatim, color, caption, faktor,esint, tikz-cd, xfrac, wrapfig, hyperref}

\newcommand{\N}{\mathbb N}
\newcommand{\R}{\mathbb R}

\newcommand{\C}{\mathbb C}
\newcommand{\Z}{\mathbb Z}

\newcommand{\mc}{\mathcal}

\newcommand{\bb}{\mathbb}

\newcommand{\Hess}{\operatorname{Hess}}

\newcommand{\Ric}{\operatorname{Ric}}

\newcommand{\GL}{\operatorname{GL}}
\newcommand{\SL}{\operatorname{SL}}

\newcommand{\SU}{\operatorname{SU}}

\newcommand{\Sp}{\operatorname{Sp}}

\newcommand{\Spin}{\operatorname{Spin}}

\newcommand{\Sym}{\operatorname{Sym}}

\newcommand{\vol}{\operatorname{vol}}
\newcommand{\tr}{\operatorname{tr}}

\newcommand{\Real}{\operatorname{Re}}

\newcommand{\Imag}{\operatorname{Im}}

\newcommand{\Stab}{\operatorname{Stab}}

\newcommand{\dil}{\operatorname{dil}}

\newcommand\halfopen[2]{\ensuremath{[#1,#2)}}

\newcommand{\bs}{\backslash}

\newtheoremstyle{break}
  {}
  {}
  {\itshape}
  {}
  {\bfseries}
  {.}
  {\newline}
  {}

  \theoremstyle{break}

\newtheorem{thm}{Theorem}[section]

\newtheorem{prop}{Proposition}[section]

\newtheorem{lemma}[prop]{Lemma}

\newtheorem*{unlabeled-prop}{Proposition}

\newtheorem{mainthm}{Theorem}

\newtheorem{conj}{Conjecture}

\theoremstyle{remark}

\theoremstyle{definition}
\newtheorem{defn}[prop]{Definition}

\numberwithin{equation}{section}

\begin{document}

\begin{abstract}
  A natural approach to the construction of nearly $G_2$ manifolds lies in resolving nearly $G_2$ spaces with isolated conical singularities by gluing in asymptotically conical $G_2$ manifolds modelled on the same cone. If such a resolution exits, one expects there to be a family of nearly $G_2$ manifolds, whose endpoint is the original nearly $G_2$ conifold and whose parameter is the scale of the glued in asymptotically conical $G_2$ manifold. We show that in many cases such a curve does not exist.

  The non-existence result is based on a topological result for asymptotically conical $G_2$ manifolds: if the rate of the metric is below $-7/2$, then the $G_2$ 4-form is exact if and only if the manifold is Euclidean $\R^7$.

  A similar construction is possible in the nearly K\"ahler case, which we investigate in the same manner with similar results. In this case, the non-existence results is based on a topological result for asymptotically conical Calabi--Yau 6-manifolds: if the rate of the metric is below $-3$, then the square of the Kähler form and the complex volume form can only be simultaneously exact, if the manifold is Euclidean $\R^6$.
\end{abstract}

\title{Topology of asymptotically conical Calabi--Yau and G\textsubscript{2} manifolds and desingularization of nearly K\"ahler and nearly G\textsubscript{2} conifolds}
\author{Lothar Schiemanowski}
\email{lothar.schiemanowski@math.uni-hannover.de}
\address{Institut f\"ur Differentialgeometrie \\ Universit\"at Hannover \\ Welfengarten 1, D--30167 Hannover\\ Germany}

\maketitle

\section{Introduction}

If $(L, g_L)$ is a closed Riemannian manifold of dimension $n$, then the cone $C(L) = (0,\infty)\times L$ can be equipped with the cone metric $g_C = dr^2 + r^2 g_L$. The manifold $L$ is called the \textit{link} of the cone. If $\Ric[g_L] = (n-1) g_L$, then $\Ric[g_C] = 0$. This relationship between positive Einstein manifolds and Ricci flat cones has a very interesting extension to the theory of special holonomy, elucidated by B\"ar in \cite{Baer}. This correspondence admits a uniform description in terms of Killing spinors and parallel spinors. It can also be described on the level of each of the holonomy groups whose underlying metric is Ricci flat:
\begin{enumerate}
\item holonomy group $\SU(n)$: the cone over a Sasaki--Einstein manifold is a Calabi--Yau cone,
\item holonomy group $\Sp(n)$: the cone over a 3-Sasaki manifold is a hyperk\"ahler manifold,
\item holonomy group $G_2$: the cone over a nearly K\"ahler manifold is a $G_2$ manifold,
\item holonomy group $\Spin(7)$: the cone over a nearly $G_2$ manfold is a $\Spin(7)$ manifold.
\end{enumerate}
In the first case the cone is of dimension $2n$, in the second case the cone has dimension $4n$, whereas the third case only appears in dimension $7$ and the last case in dimension $8$. The holonomy groups $G_2$ and $\Spin(7)$ are also known as the exceptional holonomy groups. Remarkably, the structure group $G_2$ appears both as a holonomy group and as the structure group of the geometry on the link of a $\Spin(7)$ cone. Similarly, the structure group $\SU(3)$ appears as a holonomy group and as the structure group of the geometry of the link of a $G_2$ cone. The groups $G_2$ and $\SU(3)$ are the only groups for which this is true and this observation is foundational for the investigations in this article. A remark concerning terminology is in place: nearly $G_2$ manifolds are also known as nearly parallel $G_2$ manifolds in the literature. We prefer the term ``nearly $G_2$'' to emphasize the analogy with nearly K\"ahler manifolds.

The purpose of this article is to study the possibility of constructing nearly $G_2$ and nearly K\"ahler manifolds on the basis of this coincidence.

Given any nearly K\"ahler manifold, one can construct an incomplete nearly $G_2$ manifold, which has two singularities, modelled on the $G_2$ cone associated to the nearly K\"ahler manifolds \cite{BM}. This construction is known as the sine cone construction. Spaces with singularities modelled on cones will be called conifolds in the sequel. It is conceivable that there are more examples of nearly $G_2$ conifolds.

On the other hand, for some of the known nearly K\"ahler manifolds there exist complete, non-compact $G_2$ manifolds, which are asymptotic at infinity to the cone over the nearly K\"ahler manifold. These will be called asymptotically conical $G_2$ manifolds.

A $G_2$ structure on a manifold $M$ is given by certain 3-forms $\varphi \in \Omega^3(M)$. This will be explained in more detail in section \ref{sec:SU3G2geometry}.

If $(N, \varphi_{AC})$ is a asymptotically conical $G_2$ manifold, then the rescalings $(N, t^3 \varphi_{AC})$ are also asymptotically conical $G_2$ manifolds for every $t > 0$. Moreover, $(N, t^3 \varphi_{AC})$ converges to a $G_2$ cone as $t \to 0$. Given a nearly $G_2$ conifold $(\overline{X}, \varphi_{CS})$, one can ``chop off'' the singularity and instead glue on a piece of $(N, t^3 \varphi_{AC})$. If $t$ is sufficiently small, the error made is small to zeroth order. The equations satisfied by the two pieces are quite different, however. Therefore it is a priori unclear if there exists a nearly $G_2$ manifold close to this resolved nearly $G_2$ conifold.

Suppose for the moment there exists a nearby nearly $G_2$ structure for every sufficiently small $t > 0$. Then we obtain a curve $(\overline{M}, \varphi(t))$ of nearly $G_2$ conifolds, such that $(\overline{M}, \varphi(t))$ converges to $(\overline{X}, \varphi_{CS})$ as $t \to 0$ and $(\overline{M}, t^{-3} \varphi(t))$ converges to $(N, \varphi_{AC})$. Such a family of nearly $G_2$ structures will be called a \textit{smooth desingularization}, if it also satisfies certain regularity assumptions. (See definition \ref{def:G2desing} for the precise meaning.)

The purpose of this article is to investigate whether such families can exist. It will turn out that this question is closely related to the topology of asymptotically conical $G_2$ manifolds.

\begin{mainthm}
  \label{thm:G2obstruction}
  Let $(\overline{X}, \varphi_{CS})$ be a nearly $G_2$ conifold with a singularity at $x_0$ modelled on the $G_2$ cone $(C = C(L), \varphi_C)$. Suppose $(N, \varphi_{AC})$ is an asymptotically conical $G_2$ manifold, asymptotic to $(C, \varphi_C)$.

  If there exists a smooth desingularization of $(\overline{X}, \varphi_{CS})$ by $(N, \varphi_{AC})$ at $x_0$, then $*\varphi_{AC}$ is exact.
\end{mainthm}
A precise definition of the notions of cones, conifolds and asymptotically conical manifolds will be given in section \ref{sec:cones}. Smooth desingularizations will be defined in section \ref{sec:desingularization} and this theorem and the next will also be proven there.

In six dimensions a very similar construction is possible with nearly Kähler conifolds and asymptotically conical Calabi--Yau 6-manifolds. Both classes of manifolds are described by $\SU(3)$ structures. These will be parametrized by a pair $(\omega, \Omega)$, where $\omega$ is a 2-form and $\Omega$ is a complex valued 3-form.

\begin{mainthm}
  \label{thm:SU3obstruction}
  Let $(\overline{X}, \omega_{CS}, \Omega_{CS})$ be a nearly K\"ahler conifold with a singularity at $x_0$ modelled on the Calabi--Yau cone $(C = C(L), \omega_C, \Omega_C)$. Suppose $(N, \omega_{AC}, \Omega_{AC})$ is an asymptotically conical Calabi--Yau manifold asymptotic to $(C, \omega_C, \Omega_C)$.

  If there exists a smooth desingularization of $(\overline{X}, \omega_{CS}, \Omega_{CS})$ by $(N, \omega_{AC}, \Omega_{AC})$ at $x_0$, then $\omega_{AC}^2$ and $\Real \Omega_{AC}$ are exact.
\end{mainthm}

It turns out that below a certain rate threshold asymptotically conical $G_2$ or Calabi--Yau 6-manifolds can only satisfy the conclusions of the preceding theorems if they are already Euclidean spaces. Therefore, no asymptotically conical $G_2$ manifold with rate less than $-7/2$ can desingularize a nearly $G_2$ conifold in the sense of definition \ref{def:G2desing}. Likewise, no asymptotically conical Calabi--Yau 6-manifold with rate less than $-3$ can desingularize a nearly K\"ahler manifold in the sense of definition \ref{def:NKdesing}.
\begin{mainthm}
  \label{thm:topG2}
  Let $(M,g,\varphi)$ be an asymptotically conical $G_2$ manifold of rate $\nu < -7/2$.

  Then $* \varphi \in \Omega^4(M)$ is exact if and only if $(M, g)$ is isometric to $(\R^7, g_{\mathrm{eucl}})$.

  In particular, if $(M, g)$ is not isometric to $(\R^7, g_{\mathrm{eucl}})$ the cohomology group $H^4(M, \R)$ is non-trivial.
\end{mainthm}
In 1989 Bryant and Salamon constructed the first three asymptotically conical Riemannian manifolds with holonomy group $G_2$, which were also the first complete examples with holonomy group $G_2$. \cite{BS} The underlying manifolds of these Bryant--Salamon spaces are $\Lambda^2_- S^4$, $\Lambda^2_- \bb{CP}^2$ and $S^3 \times \R^4$. The rate of $\Lambda^2_- S^4$ and $\Lambda^2_- \bb{CP}^2$ is $-4$ and to these spaces the theorem applies. On the other hand the metric on $S^3 \times \R^4$ has rate $-3$ and therefore the theorem does not apply, which is also evident from the fact that $H^4(S^3 \times \R^4) = 0$. Recently, Foscolo, Haskins and Nordstr\"om constructed more examples of asymptotically conical $G_2$ manifolds, see \cite{FHN}.

\begin{mainthm}
  \label{thm:topSU3}
  Let $(M, g, \omega, \Omega)$ be an asymptotically conical $\SU(3)$ manifold of rate $\nu < -3$.

  Then $\omega^2 \in \Omega^2(M)$ and $\Real \Omega \in \Omega^3(M)$ are simultaneously exact if and only if $(M,g)$ is isometric to $(\R^6, g_{\mathrm{eucl}})$.

  In particular, if $(M,g)$ is not isometric to $(\R^6, g_{\mathrm{eucl}})$, the cohomology group $H^2(M,\R)$ or $H^3(M, \R)$ is non-trivial.
\end{mainthm}
There exists an extensive literature for asymptotically conical Calabi--Yau manifolds, which includes existence results, see for example \cite{CH} and references included there. We just mention three cohomogeneity one examples: the Stenzel metric on $T^* S^3$, which has rate $-3$, the Candelas-DeLaOssa metric on the total space of the vector bundle $\mc{O}(-1) \oplus \mc{O}(-1) \to \bb{CP}^1$, which has rate $-2$, and the Calabi metric on the total space of the line bundle $\mc{O}(-3) \to \bb{CP}^2$, which has rate $-6$. The theorem only applies to the Calabi metric. It fails also for topological reasons in the case of the Candelas-DeLaOssa metric, since the third and fourth cohomology groups vanish for the underlying manifold. The topology of $T^* S^3$ allows the theorem to hold, but the rate $-3$ of the Stenzel metric lies just beyond the range allowed by the theorem. It is an interesting question whether a refinement of the proof would extend to the Stenzel metric. It seems that this requires a finer investigation of the leading term of the metric in appropiate coordinates.

Let us mention that the rate of asymptotically conical Ricci flat manifolds has been studied in \cite{CT} and can be estimated using analytical properties of the link, see the recent preprint \cite{KS}. In particular, asymptotically locally Euclidean Ricci flat spaces of dimension $n$ have rate $\leq -n$.

The proofs of theorems \ref{thm:topG2} and \ref{thm:topSU3} can be found in section \ref{sec:proofs}. Both proofs are very similar and they rely on a study of approximate potential functions. These are functions on asymptotically conical manifolds, which are asymptotic to $r^2/2$ and whose Laplacian is constant. Such an approximate potential can be used to define certain forms using the forms that are given by the $G_2$ or $\SU(3)$ structure. Appealing to a Hodge theoretical result shows that these forms must in fact vanish. Application of this result first requires a closer examination of the asymptotics of the approximate potential beyond the leading term. This is the technical core of the theorem and this is also where the condition on the rate comes in, see also theorem \ref{thm:approx_potential}. Representation theoretic properties of $G_2$ and $\SU(3)$ structures then imply that the vanishing of the forms implies that the trace free part of the Hessian of the approximate potential vanishes -- a theorem of Tashiro will then yield the desired statement, i.e.\@ that the asymptotically conical manifold must be isometric to a Euclidean space.  Approximate potentials have been used earlier in \cite{BH} to define an invariant of ALE gravitational instantons, which vanishes if and only if the instanton is Euclidean $\R^4$.

Theorems \ref{thm:G2obstruction} and \ref{thm:SU3obstruction} assume that there are smooth curves desingularizing the nearly $G_2$ or nearly K\"ahler manifolds. These theorems can also be interpreted as a statement about the moduli spaces of nearly $G_2$ manifolds and nearly Kähler manifolds. They indicate that nearly $G_2$ conifolds and nearly Kähler conifolds should not be thought of as boundary points of the moduli space of smooth nearly $G_2$ or nearly Kähler manifolds, except possibly in exceptional circumstances. This may be contrasted with the work of Karigiannis and Lotay \cite{KL}, who argue that in the torsion free setting the dominant part of the boundary of the moduli space of $G_2$ structures on a closed manifold should be given by $G_2$ conifolds.

Examples of Einstein desingularizations of sine cones due to B\"ohm \cite{Boehm} suggest that it can also happen that there is only a discrete sequence approaching the singular metric rather than a smooth curve. By analogy with recent results of Ozuch \cite{Ozuch1, Ozuch2, Ozuch4}, we believe that the obstructions we found also exclude such desingularizations and we therefore propose the following two conjectures.
\begin{conj}
  \label{conj:G2}
  Let $(\overline{X}, \varphi_{CS})$ be a nearly $G_2$ conifold with a singularity at $x_0$ modelled on the $G_2$ cone $(C = C(L), \varphi_C)$.

  If every asymptotically conical $G_2$ manifold asymptotic to $(C, \varphi_C)$ has rate $\nu < -7/2$, then there is no smooth nearly $G_2$ manifold Gromov--Hausdorff close to $(\overline{X}, \omega_{CS}, \rho_{CS})$.
\end{conj}
\begin{conj}
  \label{conj:NK}
  Let $(\overline{X}, \omega_{CS}, \rho_{CS})$ be a nearly K\"ahler conifold with a singularity at $x_0$ modelled on the Calabi--Yau cone $(C = C(L), \omega_C, \rho_C)$.

  If every asymptotically conical Calabi--Yau manifold asymptotic to $(C, \varphi_C)$ has rate $\nu < -3$, then there is no smooth nearly K\"ahler manifold Gromov--Hausdorff close to $(\overline{X}, \omega_{CS}, \rho_{CS})$.
\end{conj}
There are formidable analytical challenges to transferring Ozuch's approach to the weak holonomy setting.

It should be noted that the theorems and the conjectures are statements only about metrics near the conifold metrics. Indeed, the only two known inhomogeneous examples of nearly Kähler metrics, constructed by Foscolo and Haskins in \cite{FH}, \textit{do} arise as desingularizations of sine cones in a certain sense, but they do not come in families that converge back to the sine cone. Indeed, Foscolo and Haskins conjecture that the examples they construct are the only cohomogeneity one nearly Kähler metrics. From the point of view of the present article, this means that the construction should only work for one specific size of the exceptional divisor of the asymptotically conical Calabi--Yau manifold.

Previously, similar questions have been studied in the context of Einstein manifolds. Biquard investigated the possibility of desingularizing $\Z_2$ orbifold singularities in 4-manifolds in the article \cite{Biquard1} by gluing in an Eguchi--Hanson space and more general isolated orbifold singularities on 4-manifolds in \cite{Biquard2}. Biquard found a first obstruction in terms of the curvature tensor at the orbifold singularity in \cite{Biquard1} and secondary obstructions in \cite{Biquard2}. Morteza and Viaclovsky extended this analysis to $2n$-dimensional Einstein manifolds with $\Z_n$-orbifold singularities. Ozuch deepened the understanding of the Biquard obstructions in the papers \cite{Ozuch1,Ozuch2}, investigated higher order obstructions in \cite{Ozuch3} and gave a proof that certain Einstein orbifolds never appear as Gromov--Hausdorff limits of smooth Einstein manifolds in \cite{Ozuch4}, bypassing certain technical restrictions, which were present in the earlier works on the question.

In the context of special holonomy we note the works of Y.\@-M.\ Chan \cite{Chan1,Chan2}, in which (real) 6-dimensional Calabi--Yau spaces with isolated conical singularities are resolved by asymptotically conical Calabi--Yau manifolds. Similarly to our case there is a dichotomy depending on the rate of the asymptotically conical Calabi--Yau manifold; if the rate is smaller than $-3$, the problem is unobstructed. If the rate is exactly $-3$ an obstruction can be identified, which depends on a cohomology class on the link of the cone. Perhaps the article most similar to ours in subject matter is Karigiannis' \cite{Karigiannis}, in which $G_2$ conifolds are desingularized by gluing in asymptotically conical $G_2$ manifolds. In contrast to our situation, this is possible if a topological condition is verified or else the rate of the asymptotically conical $G_2$ manifolds that are glued in is less than $-4$.

\section*{Acknowledgements}
The author wishes to thank M.\ Freibert and H.\ Wei\ss\@ for many inspiring discussions, which led the author to consider the problems addressed in this article.

\section{$\SU(3)$ and $G_2$ geometry}
\label{sec:SU3G2geometry}
The use of stable forms is a convenient approach to the description of $\SU(3)$ and $G_2$ structures with special properties, which originates in \cite{Hitchin}. Let $V$ be a real, $n$-dimensional vector space. A k-form $\omega \in \Lambda^k V^*$ is \textit{stable}, if its $\GL(V)$ orbit in $\Lambda^k V^*$ is open.
The stabilizer of $\omega$ is the subgroup $\Stab_\omega = \{A \in \GL(V) : A^* \omega = \omega\}$.

Three cases are of interest to us:
\begin{enumerate}
\item If $n = 7$ and $k=3$ or $k=4$, then the stabilizer of a stable form is isomorphic to $G_2$. 
\item If $n=6$ and $k =2$ or $k=4$, then the stabilizer of a stable form is isomorphic to $\Sp(6, \R)$.
\item If $n=6$ and $k=3$, then the stabilizer of a stable form is isomorphic to either \mbox{$\SL(3,\R) \times \SL(3, \R)$} or $\SL(3,\C)$.
\end{enumerate}
We denote by $\Lambda^k_+ V^*$ the set of all stable forms for $n=7$, $k=3,4$ and for $n=6, k=2,4$. In the last case, $n = 6$ and $k = 3$, $\Lambda^3_+ V^*$ is the set of all stable forms with stabilizer isomorphic to $\SL(3,\C)$.

Each of the groups $G_2$, $\Sp(6, \R)$ and $\SL(3,\C)$ preserves a volume form in $\Lambda^n V^*$. Therefore, in these cases we can define a $\GL(V)$-invariant, $n/k$-homogeneous map $\phi = \phi_k^n : \Lambda^k_+ V^* \to \Lambda^n V^*$. Given $\omega \in \Lambda^k_+ V^*$, the derivative $d\phi[\omega]$ can be expressed by an element $\hat \omega \in \Lambda^{n-k} V^*$ via the identity $d \phi[\omega] \nu = \hat\omega \wedge \nu$. This defines the \textit{Hitchin duality map}
$$\Phi = \Phi_k^n : \Lambda^k V^* \to \Lambda^{n-k} V^*, \qquad \Phi(\omega) = \hat\omega.$$
In case $\dim V = 7$ we will denote $\Theta = \Phi_3^7$. For representation theoretic aspects of stable forms in the 7 and 6 dimensional settings we refer to \cite{Bryant} and \cite{Foscolo} respectively. The interaction between the representation theory of the 6 and 7 dimensional cases, which is needed to understand $G_2$ cones, is well explained in \cite{CS} and \cite{WW}.

\subsection{Torsion free and nearly $G_2$ forms}
Let $V$ be a $7$-dimensional real vector space and suppose $\varphi \in \Lambda^3_+ V^*$. The form $\varphi$ induces a metric $g_\varphi$ on $V$ and a Hodge dual $*_{g_\varphi}$ on $\Lambda^{*} V^*$. This can be used to describe the Hitchin dual: $\Theta(\varphi) = *_{g_\varphi} \varphi$.

Let $M$ be a 7-manifold. The fiber bundle of stable 3-forms on $M$ is given by
$$\Lambda^3_+ T^* M = \bigcup_{x\in M} \Lambda^3_+ T^*_x M$$
and the space of its sections will be denoted by $\Omega^3_+(M)$. The constructions described on a vector space in the previous paragraphs have their analogues on the manifold: an element $\varphi$ of $\Omega^3_+(M)$ induces a Riemannian metric $g_\varphi$, a volume form $\vol_\varphi$, the Hodge star $*_{g_\varphi}$. The map $\Theta$ induces a bundle map $\Lambda^3_+ T^* M \to \Lambda^4 T^*M$.
\begin{defn}
  A section $\varphi \in \Omega^3_+ (M)$ is called a \textit{$G_2$ form} on $M$. A $G_2$ form $\varphi$ is called \textit{torsion free}, if
  \begin{equation}
    \label{eq:torsionfree}
    d \varphi = 0, \qquad d\Theta(\varphi) = 0.
  \end{equation}
  A $G_2$ form $\varphi$ is called \textit{nearly parallel $G_2$ at scale $\lambda > 0$}, if
  \begin{equation}
    \label{eq:nearlypar}
    d \varphi = 4 \lambda \Theta(\varphi).
  \end{equation}
  It is called a \textit{nearly parallel $G_2$ form} if \ref{eq:nearlypar} is satisfied with $\lambda=1$.

  A manifold $M$ equipped with a torsion free $G_2$ form $\varphi$ is called a \textit{$G_2$ manifold} and is denoted by $(M, g, \varphi)$, where $g = g_\varphi$.

  A manifold $M$ equipped with a nearly parallel $G_2$ form is called a \textit{nearly $G_2$ manifold} and is denoted by $(M, g, \varphi)$, where $g=g_\varphi$.
\end{defn}
The significance of these conditions is that if $\varphi$ is a torsion free $G_2$ form, then the reduced Riemannian holonomy group of $g_\varphi$ is isomorphic to $G_2$ and if $\varphi$ is a nearly parallel $G_2$ form, then the holonomy group of its Riemannian cone is $\Spin(7)$.

\subsection{Calabi--Yau and nearly K\"ahler 6-manifolds}
Observing that $\SU(3) = \Sp(6,\R) \cap \SL(3,\C)$ allows us to parametrize a $\SU(3)$ structure on a six dimensional vector space $V$ by a stable forms $\omega \in \Lambda^2_+$ and $\Real \Omega \in \Lambda^3_+$. To ensure that the stabilizers of these forms intersect the same way as $\Sp(6,\R)$ and $\SL(3,\C)$ we impose the compatability constraints 
\begin{equation}
  \label{eq:compatibility}
  \omega \wedge \Real \Omega = 0, \qquad \frac 1 4 \Real \Omega \wedge \Phi(\Real \Omega) = \frac 1 6 \omega^3.
\end{equation}
The joint stabilizer of the pair $(\omega,\Real \Omega)$ is then isomorphic to $\SU(3)$ or $\SU(2,1)$. In addition to the compatability constraints we will assume that the joint stabilizer is isomorphic to $\SU(3)$. This is an open condition. Such a pair $(\omega,\Real \Omega)$ then induces a Riemannian metric and an almost complex structure on $V$. The complex 3-form $\Real \Omega + i \Phi(\Real \Omega)$ is a complex volume form and therefore we denote $\Imag \Omega = \Phi(\Real \Omega)$.

We will describe a $\SU(3)$ structure on $V$ by a pair $(\omega, \Omega)$, where $\omega$ is a stable 2-form and $\Omega$ is a complex 3-form, such that $\Real \Omega$ is stable and $\Imag \Omega = \Phi(\Real \Omega)$. We will denote the associated metric by $g_{\omega, \Omega}$ and the associated almost complex structure by $J_{\omega, \Omega}$.

The duality map on 2-forms can be described explicitly as $\Phi(\omega) = \frac 1 2 \omega^2$.

On a $6$-manifold $M$ one may define the bundles of stable 2-forms $\Lambda^2_+ T^*M = \bigcup_{x\in M} \Lambda^2_+ T^*_x M$ and analogously $\Lambda^3_+ T^*M$.

\begin{defn}
  A $\SU(3)$ structure on $M$ is then given by a pair $(\omega, \Omega)$, where $\omega \in \Omega^2_+(M)$ and $\Omega \in \Gamma(\Lambda^3_\C T^*M)$, such that $\Real \Omega \in \Omega^3_+(M)$, $\Imag \Omega = \Phi(\Real\Omega)$ and
  \[
  \omega \wedge \Real \Omega = 0, \qquad \frac 1 4 \Real \Omega \wedge \Imag \Omega = \frac 1 6 \omega^3.
  \]
\end{defn}

\begin{defn}
  A $\SU(3)$ structure $(\omega, \Omega)$ is \textit{torsion free} or \textit{Calabi--Yau}, if
  $$d\omega = 0, \qquad d\Real \Omega = 0, \qquad d \Imag \Omega = 0.$$
  A $\SU(3)$ structure $(\omega,\Omega)$ is \textit{nearly K\"ahler} at scale $\lambda > 0$, if
  $$d \omega = -3 \lambda \Real \Omega, \qquad d\Imag \Omega = 2 \lambda \omega^2.$$

  A 6-manifold $M$ equipped with a torsion free $\SU(3)$ structure $(\omega, \Omega)$ is called a \textit{Calabi--Yau 6-manifold} and is denoted by $(M,g,\omega, \Omega)$, where $g$ denotes the associated Riemannian metric.

  A 6-manifold $M$ equipped with a nearly Kähler $\SU(3)$ structure $(\omega, \Omega)$ of scale $1$ is called a \textit{nearly Kähler 6-manifold} and is denoted by $(M,g,\omega, \Omega)$, where $g$ denotes the associated Riemannian metric.
\end{defn}

\section{Cones, conifolds and asymptotically conical manifolds}
\label{sec:cones}
This section gives definitions of the principal objects of this article: asymptotically conical and conically singular manifolds.
\begin{defn}
  Let $L$ be a closed manifold. The \textit{cone} over $L$ is denoted by $C(L)$ and is the product manifold $(0, \infty) \times L$. The manifold $L$ is called the \textit{link} of the cone.
\end{defn}
Points in $C(L)$ are typically denoted by $(r,x)$ and $r$ can be considered as the canonical radial coordinate function. In this section, $L$ will always denote a closed manifold.
\begin{defn}
A \textit{dilation} by the factor $\lambda > 0$ on a cone $C(L)$ is the diffeomorphism
$$\dil_\lambda : C(L) \to C(L), \qquad \dil_\lambda(r,x) = (\lambda r, x).$$
  The \textit{dilation vector field} on a cone $C(L)$ is $r \partial_r$.
\end{defn}
The dilation vector field generates the dilations.

\begin{defn}
  Let $L$ be a closed manifold. A form $\kappa \in \Omega^k(C(L))$ is \textit{homogeneous of rate $\lambda$} if
  $$\mc{L}_{r\partial_r} \kappa = (k + \lambda) \kappa.$$
  Similarly, a symmetric 2-form $h \in \Gamma(\Sym^2 T^* C(L))$ is \textit{homogeneous of rate $\lambda$} if $\mc{L}_{r\partial_r} h = (2 + \lambda) h$.
\end{defn}

\begin{defn}
  A \textit{Riemannian cone} is a cone $C = C(L)$ together with a Riemannian metric $g_C$ of the form $dr^2 + r^2 g_L$, where $g_L$ is a Riemannian metric on $L$.
\end{defn}
Note that the metric $g_C$ is always homogeneous of rate $0$, but a homogeneous Riemannian metric of rate $0$ need not be of this form.

The normalization of the rates of forms and tensors is explained by the following observation. If $\kappa \in \Omega^k(C(L))$ is homogeneous of rate $\lambda$, then $|\kappa|_{g_C}$ is a homogeneous function of rate $\lambda$. In the same manner, if $h \in \Gamma(\Sym^2 T^* C(L))$ is homogeneous of rate $\lambda$, then $|h|_{g_C}$ is homogeneous of rate $\lambda$.

\begin{defn}
  Let $(L, g_L)$ be a closed Riemannian manifold. A complete Riemannian manifold $(M,g)$ is \textit{asymptotically conical} with \textit{cone at infinity} $(C, g_C)$ and \textit{rate} $\nu < 0$, if there exists a compact set $K \subset M$ and a diffeomorphism $\Phi : (R, \infty) \times L \to M \bs K$, such that for every $k$ there exists $C_k > 0$, such that
  $$r^k |\nabla^k (\Phi^* g - g_C)| \leq C_k r^\nu$$
  holds on $(R, \infty) \times L$. Here, $\nabla$ is the Levi--Civita connection of $g_C$ and $|\cdot|$ refers to the metric $g_C$.

  A \textit{radial function} on $(M,g)$ is a smooth function $\rho : M \to \R$, such that $\rho \equiv 1$ on $K$, $\rho \geq 1$ on $M$ and $\Phi^* \rho = r$ on $(R_0, \infty) \times L$ for some large enough $R_0$.
\end{defn}

\begin{defn}
Let $L$ be a closed manifold of dimension $6$ with a nearly K\"ahler structure $(\omega, \Omega)$. The \textit{$G_2$ cone} over $L$ is the manifold $C = (0,\infty) \times L$ with the $G_2$ form
$$\varphi_C = r^3 \Real \Omega - r^2 dr \wedge \omega.$$
\end{defn}
If $(C, \varphi_C)$ is a $G_2$ cone, then the dual of $\varphi_C$ is given by
$$\psi_C = \Theta(\varphi_C) = -r^3 dr \wedge \Imag \Omega - \frac 1 2 r^4 \omega^2.$$
The associated metric $g_\varphi$ is $dr^2 + r^2 g_{\omega,\Omega}$.

\begin{defn}
\label{def:ACG2}
Let $(C,\varphi_C)$ be a $G_2$ cone with link $L$. A $G_2$ structure $\varphi$ on a manifold $M$ is an \textit{asymptotically conical (AC) $G_2$ structure} asymptotic to the cone $(C, \varphi_C)$ at \textit{rate} $\nu < 0$, if there exists a compact set $K \subset M$ and a diffeomorphism $\Psi : (R, \infty) \times L \to M \bs K$, such that there exists for every $k \in \N_0$ a $C_k > 0$, for which the following inequality holds on $(R,\infty) \times L$:
$$r^k |\nabla^k_C (\Psi^* \varphi - \varphi_C)|_{g_C} \leq C_k r^{\nu}.$$

A pair $(M, \varphi)$ is called an \textit{asymptotically conical $G_2$ manifold}, if $\varphi$ is a torsion free asymptotically conical $G_2$ structure on $M$.
\end{defn}
For the sake of brevity, ``asymptotically conical'' will usually be abbreviated by AC. It is easy to generalize this definition to multiple ends, but connected AC $G_2$ manifolds automatically have only one end. An asymptotically conical $G_2$ manifold is automatically an asymptotically conical Riemannian manifold with the same rate.

In the case of conically singular manifolds, we do admit multiple conical singularities and each singularity may be modelled on a different $G_2$ cone.
\begin{defn}
\label{def:CSG2}
Let $\overline{M}$ be a topological space, $\Sigma = \{x_1, \ldots, x_n\} \subset \overline{M}$ and suppose that $M = \overline{M} \bs \Sigma$ is a smooth 7-dimensional manifold.

A \textit{conically singular $G_2$ structure} on $\overline{M}$ with singular set $\Sigma$ and cones $(C_i, \varphi_{C_i})$ of rate $\nu_i > 0$ at $x_i$, $1 \leq i \leq n$, is a $G_2$ form $\varphi$ defined on $M$, such that:
\begin{enumerate}
\item there exists a compact set $K \subset M$ and open sets $S_1, \ldots, S_n$, such that $M \bs K = S_1 \cup \ldots \cup S_n$ and the closures of the $S_i$ in $\overline{M}$ are pairwise disjoint,
\item there exist diffeomorphisms $\Psi_i : L_i \times (0, \epsilon_i) \to S_i$, such that for every $k\in \N_0$ and $i \in \{1, \ldots, n\}$ there exists a $C_{k,i} > 0$ for which the following inequality holds on $(0,\epsilon_i) \times L_i$: 
$$r^k|\nabla_{C_i}^k (\Psi_i^* \varphi - \varphi_{C_i})|_{g_{C_i}} \leq C_{k,i} r^{\nu}.$$
\end{enumerate}

A space $\overline{M}$ equipped with a conically singular $G_2$ structure will also be called a \textit{conifold with a $G_2$ structure}. If the $G_2$ structure is torsion free, we will call it a \textit{$G_2$ conifold}. If the $G_2$ structure is nearly parallel, we will call $(\overline{M}, \varphi)$ a \textit{nearly $G_2$ conifold}.
\end{defn}
Note that our terminology is slightly different from the literature, where a conifold usually also allows for AC ends.

It is clear that similar definitions are possible for manifolds with $\SU(3)$ structures. To avoid needless repetitions we leave out the precise definitions, but mention that to define a Calabi--Yau cone, one has to employ Sasaki--Einstein structures on 5-manifolds. Details on this may be found, for example, in section 2 of \cite{FH}.

\section{Smooth desingularizations and the obstruction equation}
\label{sec:desingularization}
The following definition will formalize the idea of the desingularization of nearly $G_2$ conifolds by an AC $G_2$ manifold. It can be seen as an analogue of Cheeger--Gromov convergence for a smooth 1-parameter family instead of a sequence.

Suppose $(C, \varphi_C)$ is a $G_2$ cone, $(N, \varphi_{AC})$ is an AC $G_2$ manifold with cone $(C, \varphi_C)$ and $(\overline{X},\varphi_{CS})$ is a nearly $G_2$ conifold, which has a singular point $x_0$ modelled on the cone $(C, \varphi_C)$.
\begin{defn}
\label{def:G2desing}
Suppose $\overline{M}$ is a topological space, $\Sigma \subset M$ a finite set and $\varphi(t)$, $t \in (0, \epsilon)$, is a smooth family of conically singular nearly parallel $G_2$ structures on $M = \overline{M} \bs \Sigma$. This family is a \textit{smooth desingularisation} of $(\overline{X}, \varphi_{CS})$ at $x_0$ by $(N, \varphi_{AC})$, if the following conditions are met:
\begin{enumerate}
\item For every $t \in (0,\epsilon)$ there exist open sets $U_t \subset X$, $\widetilde{U}_t \subset M$, such that
\begin{enumerate}
\item $U_s \subset U_t$ for every $s > t$,
\item $\bigcup_{t \in (0,\epsilon)} U_t = X$.
\end{enumerate}
\item For every $t \in (0, \epsilon)$ there exists a diffeomorphism $F_t : U_t \to \widetilde{U}_t$ with the following significance:
\begin{enumerate}
\item Let $\mc{U} = \bigcup_{t \in (0,\epsilon)} \{t\} \times U_t$.  The forms $F_t^*\varphi(t)$ on $U_t$ define a smooth section $\Phi \in \Gamma(\mc{U}, \pi^* \Lambda^3 T^*X)$, where $\pi : \mc{U} \to X$ is the canonical projection.
\item For any fixed $t_0$, the restriction of $\Phi$ to $U_{t_0} \times (0, t_0) \subset \mc{U}$ extends to a smooth section of $U_{t_0} \times \halfopen{0}{t_0}$ and $\Phi\big|_{\{0\} \times U_{t_0}} = \varphi_{CS}\big|_{U_{t_0}}$.
\end{enumerate}
\item For every $t \in (0, \epsilon)$ there exist open sets $V_t \subset N$ and $\widetilde{V}_t \subset M$, such that
\begin{enumerate}
\item $V_s \subset V_t$ for every $s > t$,
\item $\bigcup_{t \in (0,\epsilon)} V_t = N$.
\end{enumerate}
\item For every $t \in (0, \epsilon)$ there exists a diffeomorphism $G_t : V_t \to \widetilde{V}_t$ with the following significance:
\begin{enumerate}
\item Let $\mc{V} = \bigcup_{t \in (0,\epsilon)} \{t\} \times V_t$.  The forms $t^{-3} G_t^*\varphi(t)$ on $V_t$ define a smooth section $\hat\Phi \in \Gamma(\mc{V}, \pi^* \Lambda^3 T^*X)$, where $\pi : \mc{V} \to N$ is the canonical projection.
\item For any fixed $t_0$, the restriction of $\hat\Phi$ to $V_{t_0} \times (0, t_0)$ extends to a smooth section of $V_{t_0} \times \halfopen{0}{t_0} \subset \mc{W}$ and $\hat\Phi\big|_{\{0\} \times V_{t_0}} = \varphi_{AC}\big|_{V_{t_0}}$.
\end{enumerate}
\item For every $t$ the sets $\widetilde{U}_t$ and $\widetilde{V}_t$ cover $M$, i.e.\@ $M = \widetilde{U}_t \cup \widetilde{V}_t$.
\end{enumerate}
\end{defn}
\begin{figure}[h!]
\caption{Spaces and sets occuring in a desingularization}
\centering
\includegraphics[width=\textwidth]{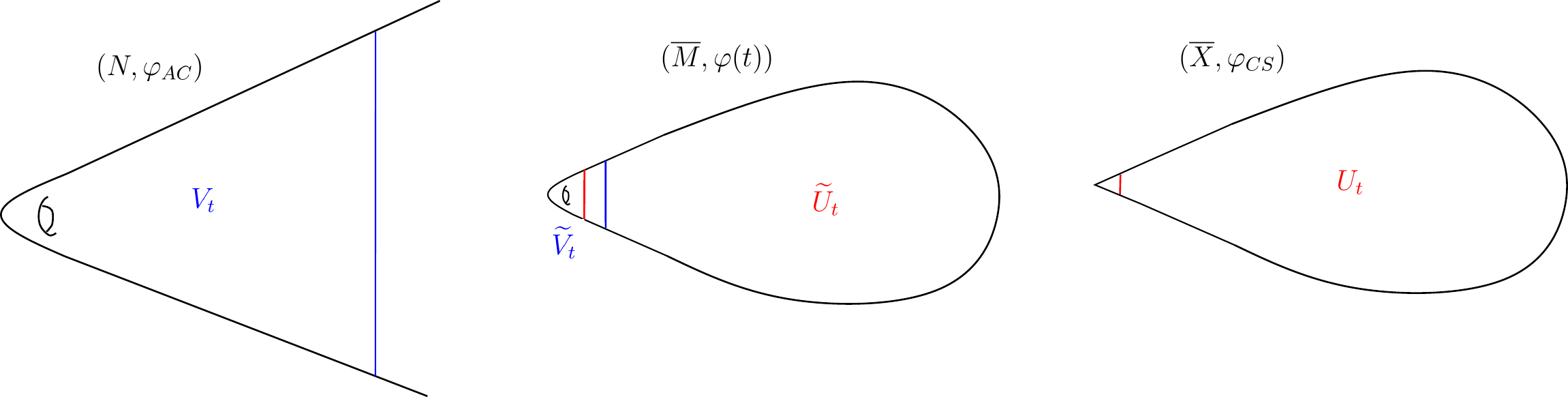}
\label{fig}
\end{figure}
Since this definition is fairly technical, we explain it in informal terms. Some of the sets appearing in the definition are illustrated in figure \ref{fig}. Roughly speaking, we want the forms $\varphi(t)$ to converge in the smooth Cheeger--Gromov sense to the conifold $(\overline{X}, \varphi_{CS})$ as $t$ goes to $0$. Similarly, a rescaling of the forms $\varphi(t)$ should converge to the AC $G_2$ manifold $(N, \varphi_{AC})$ as $t$ goes to $0$. The form $\varphi_{CS}$ is only defined on the non-compact manifold $X$ and for this reason we pick an exhausion of $X$ in (1). This is an exhaustion that expands as $t$ decreases. The second point then essentially says that $\varphi(t)$ can be pulled back to this exhaustion, so that $\varphi(t)$ defines a smooth family of forms on the ``space time'' of the exhaustion. The convergence as $t \to 0$ is then encoded as condition (2)(b), which says that the smooth family extends smoothly to the time $0$ slice. It is crucial for our application that the family depends smoothly on the parameter $t$ up to and including time $0$.

Definition \ref{def:G2desing} can easily be adapted to the nearly K\"ahler setting. In the interest of space we refrain from giving the full definition again. Instead, the following definition only serves as a reference point and to pin down the terminology.
\begin{defn}
  \label{def:NKdesing}
  Suppose $(C, \omega_C,\Omega_C)$ is a 6-dimensional Calabi--Yau cone, $(N, \omega_{AC}, \Omega_{AC})$ an asymptotically conical Calabi--Yau manifold asymptotic to $(C,\omega_C,\Omega_C)$ and $(\overline{X}, \omega_{CS}, \Omega_{CS})$ a nearly K\"ahler conifold, which has a singular point $x_0$ modelled on the cone $(C, \omega_C, \Omega_C)$. A \textit{smooth desingularization} of $(\overline{X}, \omega_{CS}, \Omega_{CS})$ at $x_0$ by $(N, \omega_{AC}, \Omega_{AC})$ is a family of conically singular nearly K\"ahler structures $(\omega(t), \Omega(t))$ satisfying the conditions (1)-(5) in \ref{def:G2desing}, changing the occurences of $\varphi(t)$ by $\omega(t)$ and $\Omega(t)$ as necessary.
\end{defn}

The existence of a smooth desingularization implies existence of a solution to an equation on the AC manifold, which we will call the obstruction equation. This will be used to prove the corollaries \ref{thm:G2obstruction} and \ref{thm:SU3obstruction}.

We start with the $G_2$ case.

Suppose $(C, \varphi_C)$ is a $G_2$ cone, $(N, \varphi_{AC})$ is an AC $G_2$ manifold with cone $(C, \varphi_C)$ and $(\overline{X},\varphi_{CS})$ is a conically singular nearly parallel $G_2$ manifold, which has a singular point $x_0$ modelled on the cone $(C, \varphi_C)$.

Let $\overline{M}$ be a topological space, $\Sigma \subset \overline{M}$ a finite space and suppose $M = \overline{M} \bs \Sigma$ is a smooth 7-manifold. Suppose $\varphi(t)$, $t\in(0,\epsilon)$, is a smooth desingularization of $(\overline{X}, \varphi_{CS})$ at $x_0$ by $(N, \varphi_{AC})$.

Then the family $\widetilde{\varphi}(t) = t^{-3} \varphi(t)$ satisfies
$$d \widetilde{\varphi}(t) = 4 t^{-3} \Theta(\varphi(t)) = 4 t \Theta(\widetilde{\varphi}(t)),$$
where we use the homogeneity $\Theta(t^3 \varphi) = t^4 \Theta(\varphi)$.

Now let $U_t$, $\mc{U}$ and $G_t$ be as in definition \ref{def:G2desing} and define $\hat{\varphi}(t) = t^{-3} G_t^*\varphi(t)$. Then $\hat\varphi(t)$ satisfies
$$d \hat\varphi(t) = 4 t \Theta(\hat{\varphi}(t))$$
on every $U_{t_0} \times \halfopen{0}{t_0}$. A priori the second equation is only satisfied on $U_{t_0} \times (0, t_0)$, but since $\hat \varphi$ is smooth on $U_{t_0} \times \halfopen{0}{t_0}$, it also holds on $U_{t_0} \times \{0\}$.

Taking the derivative of the first equation in $t$-direction yields $d \partial_t \hat\varphi(t) = 4 \Theta(\hat\varphi(t)) + 4 t d \Theta[\hat\varphi(t)] \partial_t \hat\varphi(t)$. Evaluating at time $t=0$ we obtain $d \partial_t \hat\varphi(0) = 4 \Theta(\hat\varphi(0))$.
Note that this equations hold on any $U_{t_0}$. Since $U_{t_0}$ exhausts $N$ as $t_0 \to 0$, it follows that the equation holds on all of $N$. According to definition \ref{def:G2desing} at $t=0$ the form $\hat\varphi(t)$ becomes the form $\varphi_{AC}$. Now let $\eta = \partial_t \hat \varphi(0)$. Then we have derived the \textit{obstruction equation}
$$d \eta = 4 \Theta(\varphi_{AC}) = 4 *_{g_{AC}} \varphi_{AC}.$$
Therefore $*_{g_{AC}} \varphi_{AC}$ is exact. This proves theorem \ref{thm:G2obstruction}.

The obstruction equation admits a different interpretation coming from the associated gluing problem. If one tries to glue the asymptotically conical $G_2$ manifold into the nearly $G_2$ conifold one faces two problems. The first is that the equations satisfied on each piece are obviously different. Given a solution of the obstruction equation $\eta$, we find that
$$d(\varphi_{AC} + \epsilon \eta) = 4 \epsilon \Theta(\varphi_{AC}),$$
i.e.\@ $\varphi_{AC} + \epsilon \eta$ solves the nearly $G_2$ equation at scale $\epsilon$ to first order. The second issue is that $\varphi_{AC}$ at infinity matches $\varphi_{CS}$ at the singularity only to the zeroth order. We could now derive a condition on the expansion of $\eta$ at infinity, which ensures that $\varphi_{AC} + \eta$ matches $\varphi_{CS}$ at least to the first order. As we will see later, even without this extra condition it is difficult to solve the equation $d\eta = 4 \Theta(\varphi_{AC})$ and therefore we will not pursue this here.

Note that in principle one can derive obstruction equations to any order by taking higher derivatives. However, as we will find that already the first equation is not soluble in many cases, this will also not be pursued.

This interpretation of the obstruction equations is closely related to the works of Biquard \cite{Biquard1,Biquard2}, Morteza--Viaclovsky \cite{MV}, Ozuch \cite{Ozuch1, Ozuch2, Ozuch4}. As explained in the introduction, the work of Ozuch suggests that much stronger results may be true. The interpretation of solutions of the obstruction equations as a starting point for a refined gluing construction is crucial for Biquard's and Ozuch's method of proof.

For nearly K\"ahler manifolds the derivation of the obstruction equations is very similar. To this end, we now assume that $(\overline{X}, \omega_{CS}, \Omega_{CS})$ is a nearly K\"ahler conifold with a singularity at $x_0$ modelled on the Calabi--Yau cone $(C = C(L), \omega_C, \Omega_C)$. Suppose $(N, \omega_{AC}, \Omega_{AC})$ is an asymptotically conical Calabi--Yau manifold asymptotic to $(C, \omega_C, \Omega_C)$ and suppose that $(\overline{M}, \omega(t), \Omega(t))$ is a desingularization of $(\overline{X}, \omega_{CS}, \Omega_{CS})$ at $x_0$ by $(N, \omega_{AC}, \Omega_{AC})$.

To derive the obstruction equations in this case, observe that by the construction above we obtain a family of forms $(\hat\omega(t), \hat\Omega(t))$, such that
$$d \hat\omega(t) = -3 t \Real \hat\Omega(t), \qquad d \Imag \hat\Omega(t) = 2 t \hat\omega(t)^2.$$
Taking the time derivative at $t=0$ we obtain a 2-form $\nu$ and a 3-form $\eta$ satisfying
$$d \nu = -3 \Real \Omega_{AC}, \qquad d \eta = 2 \omega_{AC}^2.$$
This proves theorem \ref{thm:SU3obstruction}. As in the $G_2$ case, we may also interpret this obstruction equation coming from the associated gluing problem.

\section{Analysis on asymptotically conical manifolds}
The proof of theorems \ref{thm:topG2} and \ref{thm:topSU3} will require an understanding of certain analytical and topological properties of AC manifolds. For the convenience of the reader we summarize the required results here.

Let $(M, g)$ be an asymptotically conical Riemannian manifold and let $\rho : M \to \R$ be a radial function.

For any tensor $T$, any integer $l \in \N_0$, $p \in \halfopen{1}{\infty}$ and $\lambda \in \R$ the $L^p_{l,\lambda}$-Sobolev norm is defined by 
$$\|u\|_{L^p_{l, \lambda}} = \left( \sum_{j \leq l} \int_M  \left| \rho^{j-\lambda} \nabla^j T\right|_g^p \rho^{-n} \vol_M \right)^{1/p}.$$
By $L^p_{l,\lambda}$ we denote the completion of the space of smooth tensors with compact support. If necessary, we will be more precise about the space of tensors considered by putting it in brackets. For example, the Sobolev space of differential $k$-forms will be denoted by $L^p_{l,\lambda}(\Lambda^k T^*M)$. The parameter $\lambda$ is called the \textit{weight} of the Sobolev space. Note that for $\lambda = -n/2$ and $p = 2$ we recover the unweighted $L^2$ space, $L^2_{0, -n/2} = L^2$.

There is a well developed theory of elliptic operators acting on these spaces. We will only need results for the Laplacian acting on functions, $\Delta_g = -\tr_g \Hess g$. It is easy to see that $\Delta_g$ is a bounded operator
$$\Delta_g : L^p_{l+2,\lambda} \to L^p_{l,\lambda-2}$$
for any choice of $p, l, \lambda$ as above. The the set of \textit{critical rates} on the cone $(C, g_C)$ is defined to be
$$\mc{D} = \left\{\lambda \in \R : \exists f \in C^\infty(C): \Delta_g \kappa = 0, \; \mc{L}_{r \partial r} f = \lambda \kappa \right \}.$$
This means that $\lambda \in \R$ is a critical rate of the Laplacian, if there exists a homogeneous harmonic function of rate $\lambda$ on the cone. This set is crucial for describing the mapping properties of $\Delta_g$, as the following theorem explains.
\begin{thm}
  \label{thm:fredholm}
  The Laplacian $\Delta_g$ has the following properties.
  \begin{enumerate}
  \item If $\lambda \notin \mc{D}$, the operator $\Delta_g : L^p_{l+2, \lambda}  \to L^p_{l,\lambda-2}$ is Fredholm.
  \item The kernel \mbox{$\ker \Delta_g : L^p_{l+2, \lambda}  \to L^p_{l,\lambda-2}$} does not depend on $p$ or $l$ and is denoted by $\ker_\lambda \Delta_g$.
  \item If $\lambda_1 < \lambda_2 \in \R$ are such that $[\lambda_1, \lambda_2] \cap \mc{D}_k = \varnothing$, then $\ker_{\lambda_1} \Delta = \ker_{\lambda_2} \Delta$.
  \end{enumerate}
\end{thm}
A close reference for this result is section 4.2 of \cite{KL}. (In this case, the theory is discussed in dimension 7, but it holds in any dimension.) This particular approach to the Fredholm theory of elliptic operators on non-compact manifolds goes back to Lockhart and McOwen, see \cite{LM}.

On a closed Riemannian manifold, one consequence of Hodge theory is that any coclosed, exact differential form must vanish. On non-compact manifolds this statement fails in general. The following theorem, due to Lockhart, gives a partial replacement for this statement on asymptotically conical Riemannian manifolds.
\begin{thm}
  \label{thm:hodge_theory}
  Suppose $(M,g)$ be an n-dimensional AC manifold.

  If $\kappa \in L^2(\Lambda^k T^*M)$ with $k \geq n/2$ is exact and coclosed, then $\kappa \equiv 0$.
\end{thm}
See \cite{Lockhart}, theorem 7.4 and example 0.15.

\section{Approximate potentials and their asymptotics}
\label{sec:approx_potential}

Let $(L, g_L)$ be a closed Riemannian manifold. The cone metric $g_C = dr^2 + r^2 g_C$ on  $C = (0, \infty) \times L$ has a \textit{potential}: the function $\frac 1 2 r^2$ satisfies
$$\Hess_{g_C} \frac 1 2 r^2 = g_C.$$
Conversely, the existence of a function whose Hessian is the metric  implies that the metric is locally conical. The only complete manifold admitting such a function is Euclidean space. This is the content of the following theorem, due to Tashiro. \cite{Tashiro}
\begin{thm}
  \label{thm:tashiro}
  Let $(M,g)$ be a complete $n$-dimensional Riemannian manifold. If there exists $u \in C^2(M)$ with
  $$\Hess_g u = g,$$
  then $(M,g)$ is isometric to $(\R^n, g_{\mathrm{eucl}})$.
\end{thm}
On an AC manifold $(M,g)$ we can study \textit{approximate potentials}, that is functions $u \in C^\infty(M)$, which are asymptotic to $\frac 1 2 r^2$ and satisfy $\Delta_g u = -n$. These functions and understanding their precise asymptotics will be fundamental for theorems \ref{thm:topG2} and \ref{thm:topSU3}. The following theorem ensures existence of such a function and describes its asymptotics.

The theorem is proven under the assumption that we have coordinates at infinity on the asymptotically manifold, in which the metric is in Bianchi gauge. On a Riemannian manifold $(M,g)$, the Bianchi operator on a symmetric 2-form is $B_g h = \delta_g h + \frac 1 2 d \tr_g h$. A metric $\widetilde{g}$ is in Bianchi gauge relative to $g$, if $B_g (\widetilde{g} - g) = B_g \widetilde{g} = 0$. By results in \cite{KS}, existence of coordinates at infinity as in the theorem may be assumed for any asymptotically conical Ricci flat manifold.
\begin{thm}
  \label{thm:approx_potential}
  Let $(M,g)$ be an asymptotically Riemannian manifold with cone at infinity $(C, g_C)$ and  rate $\nu < -2$. Suppose that $\Psi : (R, \infty) \times K \to M \bs K$ are coordinates at infinity, such that $B_{g_C} \Psi^* g = 0$ and let $\rho : M \to \R$ be a radial function.

  There exists a unique function $u \in C^\infty(M)$, such that
  $$\Delta_g u = -n, \qquad u - \frac 1 2 \rho^2 \in L^2_{2,-\epsilon}$$
  for some $\epsilon > 0$. This function satisfies
  $$u - \frac 1 2 \rho^2 \in L^2_{2, \lambda}$$
  with $\lambda = 2+\nu+\epsilon$ for every $\epsilon > 0$.
\end{thm}
\begin{lemma}
  \label{lemma:LapIso}
  Let $(M,g)$ be an n-dimensional AC manifold.

  For $\lambda < 0$ the operator
  $$\Delta_g : L^2_{2, \lambda} \to L^2_{0,\lambda-2}$$
  is an isomorphism is an isomorphism if $\lambda$ is not a critical rate. There are no critical rates in the interval $(2-n, 0)$.
\end{lemma}
\begin{proof}
  First observe that by partial integration for $\lambda < \frac 1 2 (2 - n)$ any harmonic function on $M$ is trivial and therefore $\ker_\lambda \Delta_g = 0$ for this range of weights.

  Let $(C, g_C)$ denote the cone at infinity of $(M,g)$. Let $(L,g_L)$ be its link. To see that $\lambda \in (2-n,0)$ is not a critical rate, observe that for any function $r^\lambda f$ with $f \in C^\infty(L)$ we have
  $$\Delta_{g_C} r^\lambda f = r^{\lambda-2} (\Delta_{g_L} f - \lambda (\lambda + n - 2) f).$$
  Therefore $r^\lambda f$ can only be harmonic if $\Delta_{g_L} f = \lambda (\lambda + n - 2) f$. But for $\lambda \in (2-n, 0)$ the coefficient $\lambda (\lambda + n - 2)$ is negative. On the other hand, $\Delta_{g_L}$ is a non-negative operator and therefore $f = 0$. Thus there are no non-trivial harmonic, homogeneous functions of rate $\lambda \in (2-n, 0)$.

  This also shows that $\ker_\lambda \Delta_g = 0$ for every $\lambda < 0$ by theorem \ref{thm:fredholm} (3), since $2-n < \frac 1 2 (2-n) < 0$.

  Whenever $\lambda$ is not a critical rate $\Delta_g : L^2_{2,\lambda} \to L^2_{0,\lambda-2}$ is a Fredholm operator and therefore when $\lambda < 0$ it is an isomorphism.
\end{proof}

\begin{lemma}
  \label{lemma:asymptotics}
  Let $(M,g)$ be an asymptotically Riemannian manifold with cone at infinity $(C, g_C)$ and  rate $\nu < -2$. Suppose that $\Psi : (R, \infty) \times K \to M \bs K$ are coordinates at infinity, such that $B_{g_C} \Psi^* g = 0$ and let $\rho : M \to \R$ be a radial function. Then for every $\epsilon > 0$
  $$\Delta_g \rho^2 + 2 n \in L^2_{0, \nu+\epsilon}.$$
\end{lemma}
\begin{proof}
  Denote the pullback of $g$ to the cone via $\Psi$ by $\widetilde{g} = \Psi^* g$. Then
  $$\Psi^* \Delta_g \rho^2 = \Delta_{\widetilde{g}} r^2 + 2n,$$
  because $\Psi^* \rho^2 = r^2$. Since $\Delta_{g_C} r^2 = -2n$, we may rewrite this as
  $$\Psi^* \Delta_g \rho^2 = \Delta_{\widetilde{g}} r^2 - \Delta_{g_C} r^2.$$
  To estimate the difference $\Delta_{\widetilde{g}} - \Delta_{g_C}$, recall the general formula
  $$\frac{d}{dt}\Big|_{t=0} \Delta_{g+th} u = -2 g(\Hess_g u, h) - g(B_g h, du),$$
  where $g$ is any Riemannian metric and $h$ a symmetric 2-tensor and $u$ is any function.

  For $g = g_C$ and $h = \widetilde{g} - g_C$ we therefore get
  $$\Delta_{\widetilde{g}} u = \Delta_{g_C} u - 2 g_C (\Hess_{g_C} u, h) + Q(h, du, \nabla^{g_C} du)$$
  for some $Q$ satisfying
  $$|Q(h, du, \nabla^{g_C} du)| \leq C |h|^2 ( |du|+ |\nabla^{g_C} du| ).$$
  The absence of a term depending on the zero order term $u$ is explained by the definition $\Delta_g = - \tr_g \nabla^g du$ and the fact that the operator $u \mapsto du$ does not depend on the metric.

  If we now let $u = \Psi^* \rho^2 = r^2$, we have $\Hess_{g_C} u = 2 g_C$ and since $g_C(g_C, h) = \tr_{g_C} h$, this yields
  $$\Delta_{\widetilde{g}} r^2 = \Delta_{g_C} u - 4 \tr_{g_C} h + Q(h, du, \nabla^{g_C} du).$$
  By definition of AC manifolds, we then have that there exists a constant $C_0$, such that
  $$|h|_{g_C} = |\widetilde{g} - g_C|_{g_C} \leq C_0 r^\nu.$$
  Therefore, $|\Delta_{\widetilde{g}} r^2 - 2n| \leq \hat{C}( r^\nu + r^{2\nu + 1})$ for some constant $\hat{C} > 0$. This yields the result, since $2\nu + 1 < \nu$ and $r^\nu \in L^2_{0, \nu+\epsilon}$ for every $\epsilon > 0$.
\end{proof}

\begin{proof}
  Define $p = \Delta_g \frac 1 2 \rho^2 + n$. By lemma \ref{lemma:LapIso}, the operator
  $$\Delta_{g_{AC}} : L^2_{2, \lambda} \to L^2_{0,\lambda-2}$$
  is an isomorphism for every $\lambda < 0$ that is not a critical rate.

  Lemma \ref{lemma:asymptotics} says that $p \in L^2_{0, \nu+\epsilon}$ for every $\epsilon > 0$. Since the set of critical rates is discrete, it follows that $\Delta_g : L^2_{2,\nu+\epsilon} \to L^2_{0, \nu+\epsilon}$ is an isomorphism for every $\epsilon \in (0, \epsilon_0)$ for some $\epsilon_0 > 0$, even if $\nu$ is a critical rate. Therefore $v = \Delta_g^{-1} p \in L^2_{2,\nu+2+\epsilon}$ is well defined for every $\epsilon \in (0, \epsilon_0)$.

  Now consider $u = \frac 1 2 \rho^2 - v$. Then we have
  $$\Delta_g u = \frac 1 2 \Delta_g \rho^2 - \Delta_g v = p - p - n = -n.$$
  This shows existence and the decay property.

  For uniqueness, assume that there exists another function $\hat{u}$ with $\hat{u} - \frac 1 2 \rho^2 \in L^2_{2, -\widetilde{\epsilon}}$ for some $\widetilde{\epsilon} > 0$. Let $\lambda = \max\{\nu+2+\epsilon, -\widetilde{\epsilon} \}$.  Then $\hat{u} - u \in L^2_{2, \lambda}$ and $\Delta_g (\hat{u} - u) = 0$. By lemma \ref{lemma:LapIso} it follows that $\hat{u} - u = 0$.
\end{proof}

\section{The action of $\Sym^2$ on $\Lambda^*$}

Let $(M,g)$ be a Riemannian manifold. If $h \in \Gamma(\Sym^2 T^*M)$ and $\kappa \in \Omega^k(M)$, then we define
$$h_* \kappa = -\sum_{j=1}^n h(e_j, \cdot) \wedge \iota_{e_j} \kappa,$$
where $e_1, \ldots, e_n$ is any orthonormal frame. Using the standard basis of $\Lambda^* T^*M$ induced by the dual frame $e^1, \ldots, e^n$ it is easy to see that
$$g_* \kappa = -k \kappa.$$
Recall that for any differential form $\eta \in \Omega^k(M)$ we have
$$d^* \eta = -\sum_{j=1}^n \iota_{e_i} \nabla_{e_i} \eta.$$
\begin{prop}
  \label{prop:dstarwedge}
  If $\kappa \in \Omega^k(M)$ is parallel and $u \in C^\infty(M)$, then
  $$d^* (du \wedge \kappa) = (d^*du)\kappa - (\Hess u)_* \kappa.$$
\end{prop}
\begin{proof}
  This follows from repeated use of the product rule and the definition $\Hess u = \nabla du$:
  \begin{align*}
    d^* (du \wedge \kappa) & = - \sum_{j=1}^n \iota_{e_i} \nabla_{e_i} (du \wedge \kappa)
     = - \sum_{j=1}^n \iota_{e_i} [(\nabla_{e_i} du) \wedge \kappa] \\
    & = - \sum_{j=1}^n \left[ (\iota_{e_i} \nabla_{e_i} du) \wedge \kappa - (\nabla_{e_i} du) \wedge \iota_{e_i} \kappa \right] \\
    & = (d^* du) \kappa - (\nabla du)_* \kappa.
  \end{align*}
\end{proof}

In general, not much can be said about $h_* \kappa$. For the forms associated to $G_2$ and $\SU(3)$ structures, the action of $\Sym^2$ can be described more precisely.

Suppose that $M$ is a $7$-manifold and $\varphi \in \Omega^3_+(M)$ is a $G_2$ structure. Let $g$ be the associated Riemannian metric. Then we can decompose
$$\Sym^2 T^*M = \R g \oplus \Sym^2_0 T^*M,$$
where $\Sym^2_0 T^*M = \{ h \in \Sym^2 T^* M : \tr_{g} h = 0\}$. It turns out that the map
$$\Sym^2_0 T^*M \to \Lambda^3 T^*M, \qquad h \mapsto h_* \varphi$$
is injective. This can be understood much better from the representation theoretic point of view. For this we refer to \cite{Bryant}.

If, on the other hand, $M$ is a 6-manifold and $(\omega, \Omega)$ is a $\SU(3)$ structure, then let $g$ denote the associated metric and $J$ the associated almost complex structure. In this case the space $\Sym^2 T^*M$ can be decomposed as
$$\Sym^2 T^*M = \R g \oplus \Sym^2_8 \oplus \Sym^2_{12},$$
where
$$\Sym^2_8 = \{h \in \Sym^2 T^*M : \tr_g h = 0, \; J^* h = h \},$$
$$\Sym^2_{12} = \{h \in \Sym^2 T^*M : J^* h = -h \}.$$
The subscripts denote the dimension of each of these spaces. Given $h \in \Sym^2_0 T^*M$ we may decompose
$$h = \underbrace{[h]_8}_{\in \Sym^2_8} + \underbrace{[h]_{12}}_{\in \Sym^2_{12}}.$$
With this notation, one finds that
$$h_* \omega = ([h]_8)_* \omega,$$
$$h_* \Real \Omega = ([h]_{12})_* \Real \Omega$$
and the maps
$$\Sym^2_8 \to \Lambda^2 T^*M, \qquad h \mapsto h_* \omega,$$
$$\Sym^2_{12} \to \Lambda^3 T^*M, \qquad h \mapsto h_* \Real \Omega$$
are injective. This decomposition is also best understood from the representation theoretic point of view and we refer to \cite{Foscolo} for more details.

\section{Proofs of theorems \ref{thm:topG2} and \ref{thm:topSU3}}
\label{sec:proofs}
Before we enter the proofs, let us make an observation about $G_2$ and Calabi--Yau cones.

If $\varphi$ is a torsion free $G_2$ structure, then $\varphi$ is parallel and therefore proposition \ref{prop:dstarwedge} applies. If $(C, \varphi_C)$ is a $G_2$ cone, then $\frac 1 2 r^2$ is a potential for the associated metric $g_C$. Therefore, if we apply the formula from the proposition with $u = r^2/2$, we find that
$$d^* (d r^2/2 \wedge \varphi_C) = -7 \varphi_C - g_*\varphi_C = -4 \varphi_C.$$
Applying the Hodge star to this equation shows that $*\varphi_C$ is exact.

Similarly, if $(\omega, \Omega)$ is a torsion free $\SU(3)$ structure, then $\omega$ and $\Omega$ are parallel. Then if $(C, \omega_C, \Omega_C)$ is a Calabi--Yau cone, then by considering $d^* (d r^2/2 \wedge \omega_C)$ and $d^* (d r^2/2 \wedge \Imag \Omega_C)$, we find that $*\omega_C = \frac 1 2 \omega_C^2$ and $* \Imag \Omega_C = - \Real \Omega_C$ are exact.

The idea of the proofs is to ``reverse'' this observation in a certain sense: if the $G_2$ forms or $\SU(3)$ forms on the AC manifolds are exact, then the approximate potential considered in section \ref{sec:approx_potential} must in fact be a potential function, i.e.\@ $r^2/2$.

\begin{proof}[Proof of theorem \ref{thm:topG2}]
  Suppose that $(M,g,\varphi)$ is an asymptotically conical $G_2$ manifold of rate $\nu < -7/2$.

  If $(M,g)$ is isometric to $(\R^7, g_{\mathrm{eucl}})$, then the closed form $*\varphi$ is exact by the Poincaré lemma.

  For the converse, assume that $*\varphi$ is exact. We will show that $(M,g)$ is isometric to $(\R^7, g_{\mathrm{eucl}})$.

  By theorem \ref{thm:approx_potential}, there exists a function $u : N \to \R$ with
  $$\Delta_{g_\varphi} u = -7, \qquad \Psi^* u - \frac 1 2 \rho^2 \in L^2_{2, \lambda},$$
  where $\lambda = 2 + \nu +\epsilon$ for any $\epsilon > 0$. The tracefree part of the Hessian of $u$, $\mathring{\Hess}u$, is then in $L^2_{0,\lambda-2}$. The form
  $$\kappa = * (\mathring{\Hess}u)_* \varphi$$
  is then also in $L^2_{0, \lambda-2}$. Note that for sufficiently small $\epsilon > 0$ the space $L^2_{0, \lambda-2}$ is contained in the unweighted $L^2$ space, since $\nu < -7/2$. We will now show that $\kappa$ is closed and coclosed. The assumption that $*\varphi$ is  exact will imply that $\kappa$ is exact.

  Applying proposition \ref{prop:dstarwedge} yields 
  $$d^*( du \wedge \varphi) = -7 \varphi - (\Hess u)_* \varphi.$$
  The Hessian of $u$ can be split into a tracefree part $\mathring{\Hess}u$ and its trace part,
  $$\Hess u = \mathring{\Hess} u + \frac 1 7 (\tr_g \Hess u) g = \mathring{\Hess} u + g,$$
  where we use $\tr_g \Hess u = -\Delta_g u = -7$. Since $\varphi$ is a 3-form, we have $g_* \varphi = -3 \varphi$ and therefore $(\Hess u)_* \varphi = (\mathring{\Hess} u)_* \varphi - 3 \varphi$. Therefore,
  $$(\mathring{\Hess} u)_* \varphi = 4 \varphi - d^*( du \wedge \varphi).$$
  Since $d^* \varphi = 0$ and $(d^*)^2 = 0$, we obtain $d^* (\mathring{\Hess} u)_* \varphi = 0$. This is equivalent to $d \kappa = 0$. Applying the Hodge star operator to this equation, we see that 
  $$\kappa = 4 * \varphi - d*(du \wedge \varphi)$$
  using that $*d^* = d*$ on $\Omega^4(M)$. Therefore, if $*\varphi$ is exact, then so is $\kappa$. To see that $\kappa$ is coclosed, we compute
  $$d^* \kappa =  -d^* * d^*(du \wedge \varphi) = - * d d^* (du \wedge \varphi) = - * \Delta (du \wedge \varphi),$$
  where we used that $d^* \varphi = 0$ and that $du \wedge \varphi$ is closed. Formula (10) in \cite{KL} says that $\Delta (du \wedge \varphi) = (\Delta du) \wedge \varphi$. Moreover, $\Delta du = d \Delta u = 0$, since $\Delta u$ is constant and therefore $d^* \kappa = 0$ as claimed.

  According to theorem \ref{thm:hodge_theory} coclosed, exact $L^2$ forms of form degree $\geq 4$ are trivial. Since we checked all of these conditions for $\kappa$, we conclude that $\kappa = 0$. Thus $\mathring{\Hess} u = 0$ and since $\Delta_g u = -n$, this means
  $$\Hess u = g.$$
  Therefore Tashiro's theorem \ref{thm:tashiro} implies that $(M,g)$ is isometric to $(\R^7, g_{\mathrm{eucl}})$.
\end{proof}

\begin{proof}[Proof of theorem \ref{thm:topSU3}]
  Let $(M, g, \omega, \Omega)$ be an asymptotically conical 6-dimensional Calabi--Yau manifold of rate $\nu < -3$.

  If $(M, g)$ is isometric to $(\R^6, g_{\mathrm{eucl}})$, then obviously $\omega$ and $\Imag \Omega$ are exact.

  For the other direction, we assume that $\omega^2$ and $\Imag \Omega$ are exact. We will show that then $(M, g)$ is isometric to $(\R^6, g_{\mathrm{eucl}})$.

  By theorem \ref{thm:approx_potential}, there exists a function $u : M \to \R$ with
  $$\Delta_g u = -6, \qquad \Psi^* u - \frac 1 2 \rho^2 \in L^2_{2, \lambda},$$
  where $\lambda = 2+\nu+\epsilon$ for any $\epsilon > 0$. The tracefree part of the Hessian of $u$, $\mathring{\Hess}u$, is then in $L^2_{0,\lambda-2}$. The forms
  $$\eta = *(\mathring{\Hess}u)_* \Real \Omega, \qquad \kappa = * (\mathring{\Hess}u)_* \omega$$
  are then also in $L^2_{0, \lambda-2}$. If $\epsilon$ is sufficiently small, they will also be in the unweighted $L^2$ space, since $\nu < -3$. We will now show that $\eta$ and $\kappa$ are closed and coclosed. The assumption that $\Imag \Omega$ is exact will imply that $\eta$ is exact. Likewise, the assumption that $*\omega = \frac 1 2 \omega^2$ is exact will imply that $\kappa$ is exact.

  Using proposition \ref{prop:dstarwedge} we obtain the formulas
  $$d^*(du \wedge \omega) = - 6 \omega - (\Hess u)_* \omega, \qquad d^*(du \wedge \Real \Omega) = -6 \Real \Omega - (\Hess u)_* \Real \Omega.$$
  Using these formula we may argue exactly as in the previous proof to show that $\eta$ and $\kappa$ are exact, coclosed forms.

  By theorem \ref{thm:hodge_theory} coclosed, exact $L^2$ forms of form degree $\geq 3$ are trivial. Since we checked all these conditions for $\eta$ and $\kappa$, we conclude $\eta = 0$ and $\kappa = 0$.

  If we decompose
  $$\mathring{\Hess} u = \underbrace{[\mathring{\Hess} u]_8}_{\in \Sym^2_8} +  \underbrace{[\mathring{\Hess} u]_{12}}_{\in \Sym^2_{12}},$$
  then we have
  $$(\mathring{\Hess} u)_* \omega = ([\mathring{\Hess} u]_8)_* \omega,$$
  $$(\mathring{\Hess} u)_* \Real \Omega = ([\mathring{\Hess} u]_{12})_* \Real \Omega.$$
  Moreover, since these forms are $0$, it follows that
  $$[\mathring{\Hess} u]_8 = 0 \qquad \text{and} \qquad [\mathring{\Hess} u]_{12} = 0.$$
  This shows that $\mathring{\Hess} u \equiv 0$ and by Tashiro's theorem \ref{thm:tashiro} this means $(M,g)$ is isometric to $(\R^6, g_{\mathrm{eucl}})$.
\end{proof}

\end{document}